\documentclass[letterpaper, 10 pt, conference]{ieeeconf}

\usepackage{cite}
\usepackage{amsmath,amssymb,amsfonts}
\usepackage{algorithm,algorithmic}
\usepackage{graphicx}
\usepackage{textcomp}
\usepackage{balance}
\usepackage{xcolor}
\usepackage{subcaption}

\newtheorem{theorem}{Theorem}

\newtheorem{lemma}{Lemma}

\newcommand{\by}{\mathbf{y}}

\newcommand{\ba}{\mathbf{a}}

\newcommand{\bs}{\mathbf{s}}
\newcommand{\bu}{\mathbf{u}}
\newcommand{\bv}{\mathbf{v}}
\newcommand{\bz}{\mathbf{z}}

\newcommand{\bC}{\mathbf{C}}

\newcommand{\bG}{\mathbf{G}}
\newcommand{\bK}{\mathbf{K}}

\newcommand{\bM}{\mathbf{M}}

\newcommand{\bR}{\mathbf{R}}

\newcommand{\bU}{\mathbf{U}}
\newcommand{\bV}{\mathbf{V}}

\newcommand{\br}{\mathbf{r}}
\newcommand{\bw}{\mathbf{w}}
\newcommand{\bx}{\mathbf{x}}

\newcommand{\bPhi}{\boldsymbol{\Phi}}
\newcommand{\bpsi}{\boldsymbol{\psi}}

\def\minwrt[#1]{\underset{#1}{\mathrm{minimize }}}
\def\argminwrt[#1]{\underset{#1}{\text{arg min }}}
\def\argmaxwrt[#1]{\underset{#1}{\text{arg max }}}
\def\maxwrt[#1]{\underset{#1}{\text{maximize }}}
\def\maxemphwrt[#1]{\underset{#1}{\text{\emph{maximize} }}}
\newcommand{\ett}{{\bf 1}}

\def\mminwrt[#1]{\underset{#1}{\mathrm{min }}}
\def\infwrt[#1]{\underset{#1}{\mathrm{inf }}}

\newcommand{\mR}{{\mathbb R}}

\newcommand{\cM}{\mathcal{M}}
\newcommand{\cT}{\mathcal{T}}
\newcommand{\cX}{\mathcal{X}}

\newcommand{\norm}[1]{\left\lVert#1\right\rVert}

\def\RC{{\mathbb{C}}}
\def\RR{{\mathbb{R}}}

\def\RN{{\mathbb{N}}}

\newcommand{\expop}{\mathbb{E}}
\newcommand{\expect}[1]{\expop\left(#1\right)}

\newcommand{\blambda}{\boldsymbol{\lambda}}

\newcommand{\bxi}{\boldsymbol{\xi}}

\newcommand{\fwdop}[1]{G_{#1}} 
\newcommand{\freqspace}{\mathbb{T}}  
\newcommand{\angspace}{\Theta}  
\newcommand{\measure}{\mathcal{M}_+}  
 \newcommand{\indset}[1]{[\![#1]\!]}

\newcommand\blfootnote[1]{%
  \begingroup
  \renewcommand\thefootnote{}\footnote{#1}%
  \addtocounter{footnote}{-1}%
  \endgroup
}

\newcommand{\compress}{\vspace{-4pt}}
\newcommand{\compressmini}{\vspace{-2pt}}

\begin{document}
\title{Multi-frequency tracking via group-sparse optimal transport}

\author{Isabel Haasler and Filip Elvander 
}

\maketitle

\begin{abstract}
In this work, we introduce an optimal transport framework for inferring power distributions over both spatial location and temporal frequency.
Recently, it has been shown that optimal transport is a powerful tool for estimating spatial spectra that change smoothly over time. In this work, we consider the tracking of the spatio-temporal spectrum corresponding to a small number of moving broad-band signal sources. Typically, such tracking problems are addressed by treating the spatio-temporal power distribution in a frequency-by-frequency manner, allowing to use well-understood models for narrow-band signals. This however leads to decreased target resolution due to inefficient use of the available information. We propose an extension of the optimal transport framework that exploits information from several frequencies simultaneously by estimating a spatio-temporal distribution penalized by a group-sparsity regularizer. This approach finds a spatial spectrum that changes smoothly over time, and at each time instance has a small support that is similar across frequencies. To the best of the authors’ knowledge, this is the first formulation combining optimal transport and sparsity for solving inverse problems. As is shown on simulated and real data, our method can successfully track targets in scenarios where information from separate frequency bands alone is insufficient.
\blfootnote{This work was supported by the Knut and Alice Wallenberg Foundation under grant KAW 2021.0274.}
\blfootnote{ I.~Haasler is with Signal Processing Laboratory, LTS 4, EPFL,
Lausanne, Switzerland. {\tt\small isabel.haasler@epfl.ch} }
\blfootnote{F.~Elvander is with the Dept. of Information and Communications Engineering, Aalto University, Finland. {\tt\small filip.elvander@aalto.fi} }

\end{abstract}

\section{Introduction}

Spectral estimation appears in a variety of control and signal processing applications, ranging from fault detection \cite{trachi2016induction} to noise reduction and speech enhancement \cite{Gannot2017}. For wide-sense stochastic processes, commonly employed as a signal model, the (temporal) spectrum parametrizes the covariance function and describes the distribution of power over frequency \cite{Grenander1958,ElvanderK23_71}. Analogously, in multi-sensor or array processing scenarios, the spatial spectrum gives the distribution of power on the spatial domain and parametrizes the array or spatial covariance matrix \cite{VanTrees2002}.  Commonly in applications, it is assumed that such spatial spectra correspond to temporal signals supported on a single carrier frequency, allowing for representing time-delays by wave-form phase-shifts. Although such narrow-band assumptions do not hold for scenarios with broad-band sources, approximations can be constructed by means of filter banks or the short-time Fourier transform \cite{Boehme1986}. However, as the resulting narrow-band signals are typically processed independently, information common for sets of frequencies, such as spectral coherence, is ignored \cite{WeissK02_dsp}.

In this letter, we consider the problem of spatial spectral estimation for broad-band sources or targets. Furthermore, we are interested in the case when the scene is observed at a sequence if time instances, between which the location of the targets change in a smooth fashion. In recent works, we have developed a framework for spectral estimation building on the concept of optimal transport (OT)
\cite{elvander2020multi,elvander2018tracking}.
In this setting the geometric property of OT to capture smooth shifts in distributions, i.e., spectral energy content, has shown to be a powerful tool for target tracking.
OT has also found applications in various other control applications \cite{chen2021optimal}, e.g., in control and estimation for multi-agent systems \cite{krishnan2018distributed, haasler2021control}, and uncertainty quantification \cite{aolaritei2022uncertainty}. 

In the case of spatial spectral estimation, 
our previous work has been limited to narrow-band scenarios, i.e., with signals supported on a single carrier frequency.
In this work, we develop these concepts further, and in particular to scenarios with broad-band sources. Here, all available information, corresponding to several observation times and temporal frequencies, is used in estimation of the spatio-temporal spectrum, i.e., a distribution of power over both spatial domain and temporal frequency. We propose to achieve this information sharing between frequencies by imposing the assumption of spatial sparsity: as signal sources should be relatively few, the support of each spatial spectrum should be small. Furthermore, we present an efficient algorithm with linear convergence rate implementing our proposed estimator. \looseness=-1

To the best of the authors' knowledge, this is also the first time in which OT and sparsity-inducing penalties are used jointly for solving inverse problems.

\section{Background}

\subsection{Spatio-temporal estimation}
Consider a scenario in which a superposition of broad-band signals, emitted by a set of spatially localized sources in the far-field, impinge on an array of $Q \in \RN$ sensors. Let the corresponding sensor array signal be
\begin{align*}
	\by(t) = \begin{bmatrix} y_1(t) & \ldots & y_Q(t)\end{bmatrix}^T \in \RC^Q \;,\; t \in \RR.
\end{align*}
Then, modeling the sources as wide-sense stationary processes, we seek a spatio-temporal spectrum describing the distribution of signal power over look-angle\footnote{To simplify the exposition, but without loss of generality, we here let the spatial domain correspond to direction-of-arrival (DoA).} and temporal frequency. That is, letting $\angspace$ and $\freqspace$ denote the angle and frequency spaces respectively, we seek \mbox{$\Phi \in \measure(\angspace \times \freqspace)$}. Furthermore, consider passing each sensor signal $y_k(t)$ through a filter bank\footnote{Equivalently, this can be performed as a decomposition using the short-time Fourier transform as is common in audio signal processing.} of $F$ narrow-band filters with center frequencies $\omega_f$, $f \in \indset{F} = \{1,\ldots,F\}$, yielding the set of narrow-band sensor signals
\begin{align*}
	\by_f(t) = \begin{bmatrix} y_{1,f}(t) & \ldots & y_{Q,f}(t)\end{bmatrix}^T \in \RC^Q \;,\; f \in \indset{F}.
\end{align*}
With this, the so-called spatial covariance matrix for carrier frequency $\omega_f$ is given by
\begin{align}\label{eq:measurement_eq}
	\bR_f \!\triangleq \!\expect{\by_f\by_f} \!=\! \int_{\angspace} \ba_f(\theta) \ba_f(\theta)^H \Phi_f(\theta)d\theta = \fwdop{f}(\Phi_f)
\end{align}
where $\expect{\cdot}$ denotes the expectation operation, and where we have defined the set of linear operators $\fwdop{f}: \cM_+(\angspace) \to \RC^{Q\times Q}$. Here, the vector functions $\ba_f: \angspace \to \RC^{Q}$ denotes the array response at carrier frequency $\omega_f$, encoding the array geometry, as well as filter response and propagation properties of the space. Then, given $\bR_f$, $f \in \indset{F}$, or estimates thereof, we seek to estimate $\Phi$, or more precisely  $\{ \Phi_f \}_{f=1}^F = \{ \Phi(\omega_f,\cdot) \}_{f=1}^F$.

\subsection{Spectral tracking via  optimal transport}

Let $\Phi_0, \Phi_1 \in \cM_+(X)$ be two non-negative distributions, i.e., generalized functions, on a space $X$.
The optimal transport problem \cite{peyre2019computational, villani2021topics} is to transform $\Phi_0$ into $\Phi_1$ in the most efficient way, where efficiency is measured in terms of a cost function $c:X\times X \to \mR$, and where $c(x_0,x_1)$ denotes the cost for moving a unit mass from $x_0\in X$ to $x_1\in X$.
OT finds a so-called transport plan, which is a bi-variate distribution $m \in \cM_+(X\times X)$, that minimizes
\begin{equation} \label{eq:ot} 
\begin{aligned}
S(\Phi^{(0)}\!,\Phi^{(1)} )  = \!\!\!\! \mminwrt[m \in \cM_+(X\times X)] & \int_{X\times X} \!\!\!\! c(x_0,x_1) m(x_0,x_1) dx_0 dx_1 \\
\text{subject to }\ & \int_X m(x_0,x_1) dx_1 = \Phi^{(0)}(x_0)\\
& \int_X m(x_0,x_1) dx_0 = \Phi^{(1)}(x_1)
\end{aligned}
\end{equation}
The objective value of \eqref{eq:ot} can be interpreted as a measure of distance between the distributions $\Phi^{(0)}$ and $\Phi^{(1)}$, quantifying how much the mass has to be moved in order to transform $\Phi^{(0)}$ into $\Phi^{(1)}$.
This property has recently proven useful in the setting of tracking spatial spectra based on covariance measurements as in \eqref{eq:measurement_eq}.
Namely, with $X = \angspace$, in \cite{elvander2018tracking, elvander2020multi} the OT distance \eqref{eq:ot} is used as a regularizing term to find spectral estimates whose mass moves smoothly between consecutive time points, which results in the formulation 
\begin{equation} \label{eq:tracking_bimarginal}
	\minwrt[\Phi^{(t)}, t \in\indset{\cT} ]\;   \sum_{t=1}^\cT \|  \fwdop{} (\Phi^{(t)}) - \bR^{(t)} \|_2^2 + \sum_{t=1}^{\cT-1}\! S(\Phi^{(t)},\Phi^{(t+1)} ).
\end{equation}
Moreover, the problem can equivalently be posed as a so-called multi-marginal OT problem over the product space $\cX = \angspace^\cT := \angspace\times \dots \times \angspace$  \cite{elvander2020multi}.
In this setting, $c$ and $m$ are a $\cT$-variate function and distribution, respectively, where $c(x)$ and $m(x)$ denote the cost and amount of transport associated with a tuple $x=(x_1,\dots,x_T)$.
The multi-marginal formulation of \eqref{eq:tracking_bimarginal} reads
\begin{equation} \label{eq:tracking_multi}
	\minwrt[m \in \cM_+(\cX)]  \int_\cX \!\!\! m(x) c(x) dx + \sum_{t=1}^\cT \| \fwdop{} (P^{(t)}(m)) - \bR^{(t)} \|_2^2,
\end{equation}
where the cost function decouples into pairwise interactions,
\begin{equation} \label{eq:cost_seq}
	c(x_0,\dots,x_T) = \sum_{t=1}^{\cT-1} c(x_t,x_{t+1}),
\end{equation}
and $P^{(t)}(m) \in \cM_+(X)$ denotes projections of the transport plan, defined as
\begin{equation*}
P^{(t)}(m) = \int_{X^{T-1}} \!\!\!\! m(x_1,\dots,x_\cT) dx_0 \dots dx_{t-1} dx_{t+1} \dots dx_T.
\end{equation*}
Analogously to standard OT problems \cite{cuturi13sinkhorn, benamou15iterative}, an approximate solution to \eqref{eq:tracking_multi} can be found by adding an entropic regularization term to the discretized problem \cite{elvander2020multi}.
Note that although discretizing the multi-marginal optimization problem \eqref{eq:tracking_multi} results in a much larger optimization problem than discretizing \eqref{eq:tracking_multi}, it turns out that utilizing the structure in the cost \eqref{eq:cost_seq} reduces the computational complexity drastically to the same order of operations.
Moreover, this approach results in sharper estimates of the distributions, which is a desirable properties in many scenarios, including DoA estimation \cite{elvander2020multi, haasler21multimarginal}.

\section{Problem formulation}

In this work we consider the setting where a small number of targets are emitting broad-band signals that we measure at several different frequencies $\omega_1,\dots,\omega_F$. Thus, the spatial power spectra $\Phi_f(\theta)$, for $f \in \indset{F}$, to be estimated are expected to have supports concentrated on a small set of angles. Furthermore, as the targets are broad-band, these supports are similar across frequency.
Herein, we propose to model this by requiring that the spatial sparsity measure
\begin{equation} \label{eq:gs_reg}
\int_{\angspace} \sup_{f \in \indset{F}}  \left| \Phi_f(\theta) \right|  d\theta
\end{equation}
is small.
In order to find power spectra that are spatially sparse and change smoothly over time, we propose to combine the tracking formulation \eqref{eq:tracking_multi} with the group-sparsity regularizer \eqref{eq:gs_reg}.
More precisely, we seek a transport plan $m_f\in \mathcal{M}_+(\cX)$ for each frequency $f \in \indset{F}$. Note that its projections describe the power spectra at the discrete time instances, $P^{(t)}(m_f) = \Phi_{f}^{(t)}$ for $f \in \indset{F}$ and $t \in \indset{\cT}$.
Moreover, let $\fwdop{f}$ and $\bR_{f}^{(t)}$ denote the measurement operator and covariance measurements at frequency $f$ and time $t$.
Then, we formulate the tracking problem for group-sparse spectra as
\begin{equation} \label{eq:probform}
\begin{aligned}
	\minwrt[m_f \in \cM_+(\cX)] \  &  \sum_{f=1}^F \sum_{t=1}^\cT \| \fwdop{f} (P^{(t)}(m_f)) - \bR_{f}^{(t)} \|_2^2 \\
	& + \alpha 
	\sum_{f=1}^F \int_\cX m_f(x) c(x) dx \\
	& + \beta 
	\sum_{t=1}^\cT \int_{\angspace} \sup_{f \in \indset{F}} \left\{ P^{(t)}(m_f) \right\} d\theta,
\end{aligned}
\end{equation}
where the cost is structured as in \eqref{eq:cost_seq}, and $\alpha, \beta > 0$ 
are parameters that regulate the emphasis on smoothness over time and group-sparsity, respectively.
\subsection{Discretization and entropic regularization}
Following previous works \cite{benamou15iterative, elvander2020multi, haasler21multimarginal}, we solve \eqref{eq:probform} by discretizing it and regularizing it with an entropic term.
We discretize the angle space $\angspace$ into $N$ grid points $\theta_1,\dots,\theta_N$.
The cost and transport plans are then described by $\cT$-mode tensors $\bC \in \RR^{N\times \ldots \times N}$ and $\bM_f \in \RR_+^{N\times\ldots \times N}$, where the elements are defined as $\bC_{i_1,\dots,i_\cT} = c(\theta_{i_1},\dots, \theta_{i_\cT})$ and similarly for $\bM_f$.
The discrete projection operator is defined as
\begin{equation*}
	[ P^{(t)}(\bM) ]_{i_t} = \sum_{i_1,\dots,i_{t-1},i_{t+1},\dots,i_{\cT}} \bM_{i_1,\dots,i_{t-1},i_t, i_{t+1},\dots,i_{\cT}},
\end{equation*}
and $P^{(t)}(\bM) \in \RR^N$ is the discretization of the spectrum $\Phi_{f}^{(t)}$.
Thus, the discrete version of the group-sparsity term \eqref{eq:gs_reg} reads
\begin{equation*}
	\norm{\begin{bmatrix}P^{(t)}(\bM_1),\!\ldots\!,P^{(t)}(\bM_F)\end{bmatrix}}_{\infty,1} \!=\! \sum_{i=1}^N \!\sup_{f \in \indset{F}}\! \big| [P^{(t)}\!(\bM_f)]_i \big|.
\end{equation*}
Moreover, let 
\begin{equation*}
\br_f^{(t)} =\begin{bmatrix}  \mathrm{vec}\left(\mathrm{Re}(\bR_f^{(t)})\right)^T & \mathrm{vec}\left(\mathrm{Im}(\bR_f^{(t)})\right)^T \end{bmatrix}^T
\end{equation*}
and let $\bG_f$ be the discrete counterpart of $\fwdop{f}$.
This lets us formulate the discretized and regularized group-sparse OT problem
\begin{equation} \label{eq:ot_discreg}
\begin{aligned}
	\minwrt[\bM_f \in \RR^{N \times \ldots \times N }]&\quad \sum_{f=1}^F \langle \bC, \bM_f\rangle + \epsilon D(\bM_f) 
	\\&+ \sum_{t=1}^T \sum_{f= 1}^F \frac{\gamma}{2} \norm{\br_f^{(t)} - \bG_f P_t(\bM_f)}_2^2
	\\&+ \eta \sum_{t=1}^{T} \norm{\begin{bmatrix}P^{(t)}(\bM_1),\ldots,P^{(t)}(\bM_F)\end{bmatrix}}_{\infty,1},
\end{aligned}
\end{equation}
where we for convenience of the following exposition set $\gamma = 2/\alpha$, $\eta=\beta/\alpha$, and
\begin{equation*}
 D(\bM) = \!\! \sum_{i_1,\ldots,i_\cT=1}^{N} \!\!\!\! \big( \bM_{i_1,\ldots,i_\cT} \log(\bM_{i_1,\ldots,i_\cT}) - \bM_{i_1,\ldots,i_\cT} + 1 \big)
\end{equation*}
is an entropic regularization term, and $\epsilon>0$ is a small regularization parameter.

\section{Method}

We solve the discretized and regularized problem \eqref{eq:ot_discreg} by a dual block coordinate descent, following the approach in \cite{benamou15iterative, elvander2020multi, haasler21multimarginal}.

\subsection{Dual problem}

\begin{theorem} \label{thm:dual}
The unique optimal transport plans $\bM_f$, $f \in \indset{F}$, are represented as
\begin{equation} \label{eq:M}
	\bM_f = \bU_f \odot \bK \odot \bV_f
\end{equation}
where
\begin{align*}
	\bU_f &= \bu_f^{(1)} \otimes \ldots \otimes \bu_f^{(T)} \text{ with } \bu_f^{(t)} = \exp\left( \frac{1}{\epsilon} \bG_f^T\blambda_f^{(t)}\right)
	\\ \bV_f  &= \bv_f^{(1)} \otimes \ldots \otimes \bv_f^{(T)} \text{ with } \bv_f^{(t)} = \exp\left( \frac{1}{\epsilon} \bpsi_f^{(t)}\right) 
	\\ \bK &= \exp\left(-\frac{1}{\epsilon}\bC\right)
\end{align*}
and where $\blambda_f^{(t)}$ and $\bpsi_f^{(t)}$ solves the dual problem
\begin{equation} \label{eq:dual}
\begin{aligned}
	\minwrt[\blambda_f^{(t)}, \bpsi_f^{(t)}] &\quad \sum_{f= 1}^F\epsilon \left\langle \bU_f , \bV_f \odot  \bK \right\rangle
	\\&\quad + \sum_{t=1}^{\cT} \sum_{f=1}^F \left( \frac{1}{2\gamma}\norm{\blambda_f^{(t)}}_2^2 -\langle\blambda_f^{(t)}, \br_f^{(t)}\rangle \right)
	\\ \text{subject to } &  \norm{ \begin{bmatrix}\bpsi_1^{(t)},\ldots,\bpsi_F^{(t)}\end{bmatrix}}_{1,\infty} \leq \eta \;,\; t \in \indset{\cT},
\end{aligned}
\end{equation}
where $\norm{\cdot}_{1,\infty}$ is the dual norm of $\norm{\cdot}_{\infty,1}$.
\end{theorem}

\begin{proof}
See appendix.
\end{proof}
We note that the first term in the objective of the dual \eqref{eq:dual} can be expressed in terms of the components in \eqref{eq:M}.
\begin{lemma} \label{lem:xi}
With the decomposition \eqref{eq:M} and cost function according to \eqref{eq:cost_seq}, we can write
\begin{equation*}
\langle \bU_f, \bV_f \odot \bK \rangle = \| \bu_f^{(t)} \odot \bv_f^{(t)} \odot \bxi_f^{(t)} \|_1,
\end{equation*}
where $\bxi_f^{(t)} = \bw_f^{(t)} \odot \hat{\bw}_f^{(t)}$, and
\begin{equation} \label{eq:what}
\hat{\bw}_f^{(t)} \leftarrow  \begin{cases} \ett & ,t = 1 \\ {\bK}^{T}  \left(\bu_f^{(t-1)} \odot \bv_f^{(t-1)} \odot \hat{\bw}_f^{(t-1)}\right) &, t> 1 \end{cases},
\end{equation}
\begin{equation} \label{eq:w}
\bw_f^{(t)} \leftarrow \begin{cases}{\bK}^{}\left(\bu_f^{(t+1)} \odot \bv_f^{(t+1)} \odot \bw_f^{(t+1)}\right) & , t < \cT \\ \ett &, t = \cT\end{cases}. \!\!
\end{equation}
\end{lemma}
\begin{proof}
The results follows similarly to \cite[Proof of Proposition~2]{elvander2020multi}.
\end{proof}

\subsection{Algorithm}
We propose to solve \eqref{eq:ot_discreg} by the means of a block coordinate descent in its dual problem \eqref{eq:dual}.
That is, we iteratively optimize \eqref{eq:dual} with respect to one set of variables, while keeping the other variables fixed.
More precisely, we iterate the following steps
\begin{enumerate}
\item For $t\in \indset{\cT}$ and $f \in \indset{F}$ let
\begin{equation} \label{eq:update_lambda}
\hspace{-20pt} \blambda_f^{(t)} \leftarrow \argminwrt[ \blambda] \  \epsilon \langle e^{\bG_f^{(t)}\blambda/\epsilon}, \bv_f^{(t)} \odot \bxi_f^{(t)}\rangle + \frac{1}{2\gamma}\norm{\blambda}_2^2 -\langle\blambda, \br_f^{(t)}\rangle
\end{equation}
\item For $t \in \indset{\cT}$ let 
\begin{equation} \label{eq:update_psi}
\begin{aligned}
\hspace{-20pt} \{\bpsi_1^{(t)},\dots,\bpsi_F^{(t)} \} \leftarrow	\minwrt[ \bpsi_1,\dots,\bpsi_F] &\quad \sum_{f= 1}^F \epsilon \langle e^{\bpsi_f/\epsilon}, \bu_f^{(t)} \odot  \bxi_f^{(t)}\rangle
	\\ \text{subject to } &  \norm{ \begin{bmatrix}\bpsi_1,\ldots,\bpsi_F\end{bmatrix}}_{1,\infty} \leq \eta .
\end{aligned}
\end{equation}
\end{enumerate}
Note that before each step the vector $\bxi_f^{(t)}$ must be computed as described in Lemma~\ref{lem:xi}.
Since iteratively updating the dual variables according to \eqref{eq:update_lambda} and \eqref{eq:update_psi} is a block coordinate descent, and the dual problem \eqref{eq:dual} satisfies the assumptions of \cite[Theorem~2.1]{luo1992convergence}\footnote{The theorem requires standard assumptions on the optimization problem, e.g., strict convexity.}, the iterates converge linearly to the optimal solution of \eqref{eq:dual}.
In the limit point, the optimal solution to the primal \eqref{eq:ot_discreg} can be constructed as described in \eqref{eq:M}.
It turns out that the optimization problems \eqref{eq:update_lambda} and \eqref{eq:update_psi} can be solved efficiently.
First, we note that the minimizer $\blambda$ of \eqref{eq:update_lambda} solves 
\begin{equation} \label{eq:lambda_cond}
	\bG_f \exp\left(\frac{1}{\epsilon} \bG_f^T\blambda \right) \odot \bv_f^{(t)} \odot  \bxi_f^{(t)}+ \frac{1}{\gamma}\blambda - \br_f^{(t)} = 0,
\end{equation}
which we solve by a Newton's method, as proposed in \cite{elvander2020multi}.
Secondly, \eqref{eq:update_psi} can be solved as described in the following.
\begin{theorem} \label{thm:psi}
The solution $\{\bpsi_1^{(t)},\dots,\bpsi_F^{(t)}\}$ of \eqref{eq:update_psi} can be constructed for each index $i=1,\dots,N$ separately by performing the following steps.
\begin{enumerate}
	\item Let $\bz = - \begin{bmatrix} (\bu_1^{(t)})_i (\bxi_1^{(t)})_i,\dots, (\bu_F^{(t)})_i  (\bxi_F^{(t)})_i \end{bmatrix}^T \in \RR^F$.
	\item Sort the vector $\bz$ in ascending order.
	\item Identify $f^*$ such that $g(z_{f^*}) > 0$ and $g(z_{f^*+1}) \leq 0$, where
		\begin{equation}\label{eq:g_func}
		g(z) \triangleq -(F-k)\log z - \frac{\eta}{\epsilon} + \sum_{\ell = k+1}^F \log z_\ell .
		\end{equation}
	\item For $f \in \indset{F}$, let 
	\begin{equation*}
	\hspace{-25pt} ( \bpsi_f^{(t)} )_i   = - \epsilon \ \text{max}\bigg( 0, \  \log z_f + \frac{1}{F\!-\!f^*}  \Big( \frac{\eta}{\epsilon} - \sum_{\ell = f^*+1}^F \!\! \log z_\ell \Big) \! \bigg).
	\end{equation*}
\end{enumerate}
\end{theorem}
\begin{proof}
See appendix.
\end{proof}

The full method is summarized in Algorithm~\ref{alg:bcd}.
\begin{algorithm} [tb]
\begin{algorithmic}
\WHILE{Not converged}
	\FOR {$t \in \{1,\ldots,\cT\}$}
		\IF{forward sweep}
		\STATE Update $\hat{\bw}_f^{(t)}$ for $f \in \indset{F}$ as in \eqref{eq:what}
		\ELSE
			\STATE Update $\bw_f^{(t)}$ for $f \in \indset{F}$ as  in \eqref{eq:w}
		\ENDIF
		\STATE $\bxi_f^{(t)} \leftarrow \bw_f^{(t)} \odot \hat{\bw}_f^{(t)}$ , $f \in \indset{F}$
		\STATE $\blambda_f^{(t)} \leftarrow $ solution to \eqref{eq:lambda_cond}, $f \in \indset{F}$
		\STATE $\bu_f^{(t)} \leftarrow \exp\left(\frac{1}{\epsilon} \bG_f^T\blambda_f^{(t)} \right)$, $f \in \indset{F}$
		\STATE $\bpsi_f^{(t)}, f \in \indset{F} \  \leftarrow$ as in Theorem~\ref{thm:psi},  
		\STATE $\bv_f^{(t)} \leftarrow \exp\left(\frac{1}{\epsilon} \bpsi_f^{(t)} \right)$, $f \in \indset{F}$
	\ENDFOR
\ENDWHILE
\end{algorithmic}
\caption{Group-sparse multi-marginal OT} \label{alg:bcd}
\end{algorithm}
The algorithm sweeps forward and backwards through the time index $t \in \indset{\cT}$. By storing previous results for the vectors $\hat{\bw}_f^{(t)}$ and $\bw_f^{(t)}$, in each iteration only one of these vectors has to be updated for all $f \in \indset{F}$. This requires $F$ matrix-vector multiplications as in \eqref{eq:what}- \eqref{eq:w}, where the matrix is of size $N\times N$, and is thus of complexity $\mathcal{O}(FN^2)$.
The update of $\lambda_f^{(t)}$ requires finding the root of \eqref{eq:lambda_cond} by Newton's method. We observe that after a few outer Sinkhorn iterations, the Newton method typically converges within one step, and thus requires solving only one system of $Q^2$ linear equations. 
Finally, for the updates of $\bpsi_f^{(t)}$ we need to perform the steps listed in Theorem~\ref{thm:psi} for the $N$ elements in $\bpsi_f^{(t)}$. The most computationally expensive operation here is the sorting in step 2) which requires $\mathcal{O}(F\log F)$ operations. One update of $\bpsi_f^{(t)}$ for a given $t$ and $f \in \indset{F}$ has thus complexity $\mathcal{O}(NF\log F)$. \looseness=-1

\begin{figure}[t]
\centering
%
\begin{subfigure}[t]{0.23\textwidth}
        \centering
            \includegraphics[width=\textwidth]{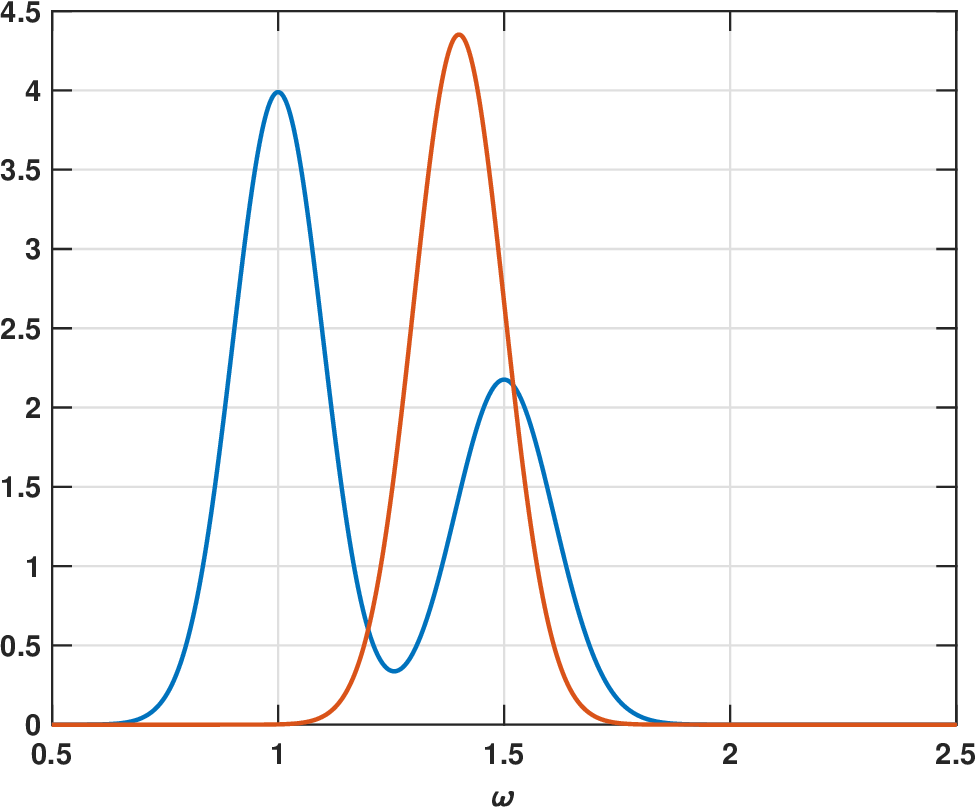}
            \caption{Temporal spectra.}
            \label{fig:temporal_spectrum}
\end{subfigure}
\begin{subfigure}[t]{0.23\textwidth}
        \centering
            \includegraphics[trim={50pt 30pt 50pt 0pt}, clip, width=\textwidth]{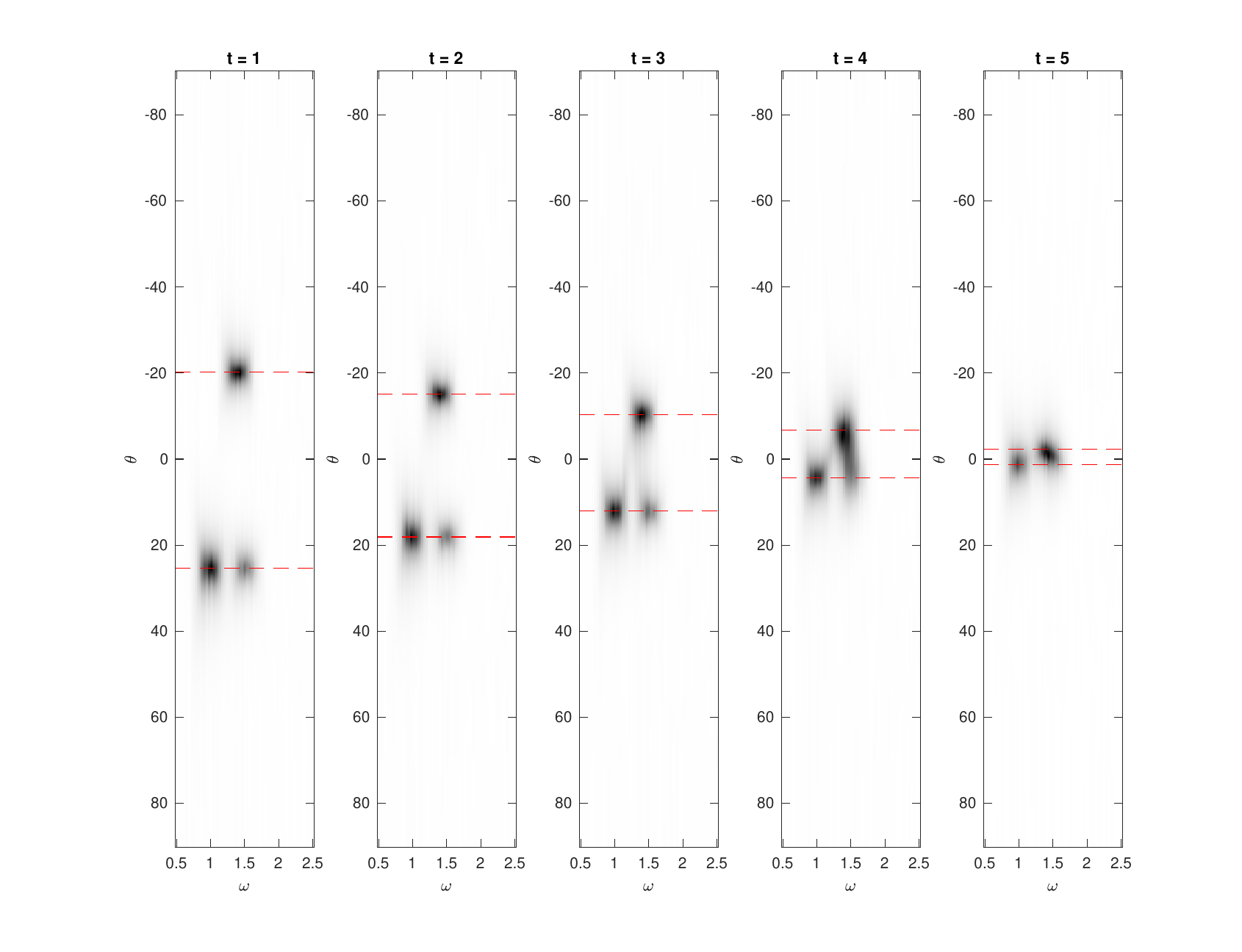}
           \caption{MVDR estimate.}
            \label{fig:capon_spectrum}
\end{subfigure}
%
\begin{subfigure}[t]{0.23\textwidth}
        \centering
            \includegraphics[trim={40pt 20pt 40pt 15pt}, clip, width=\textwidth]{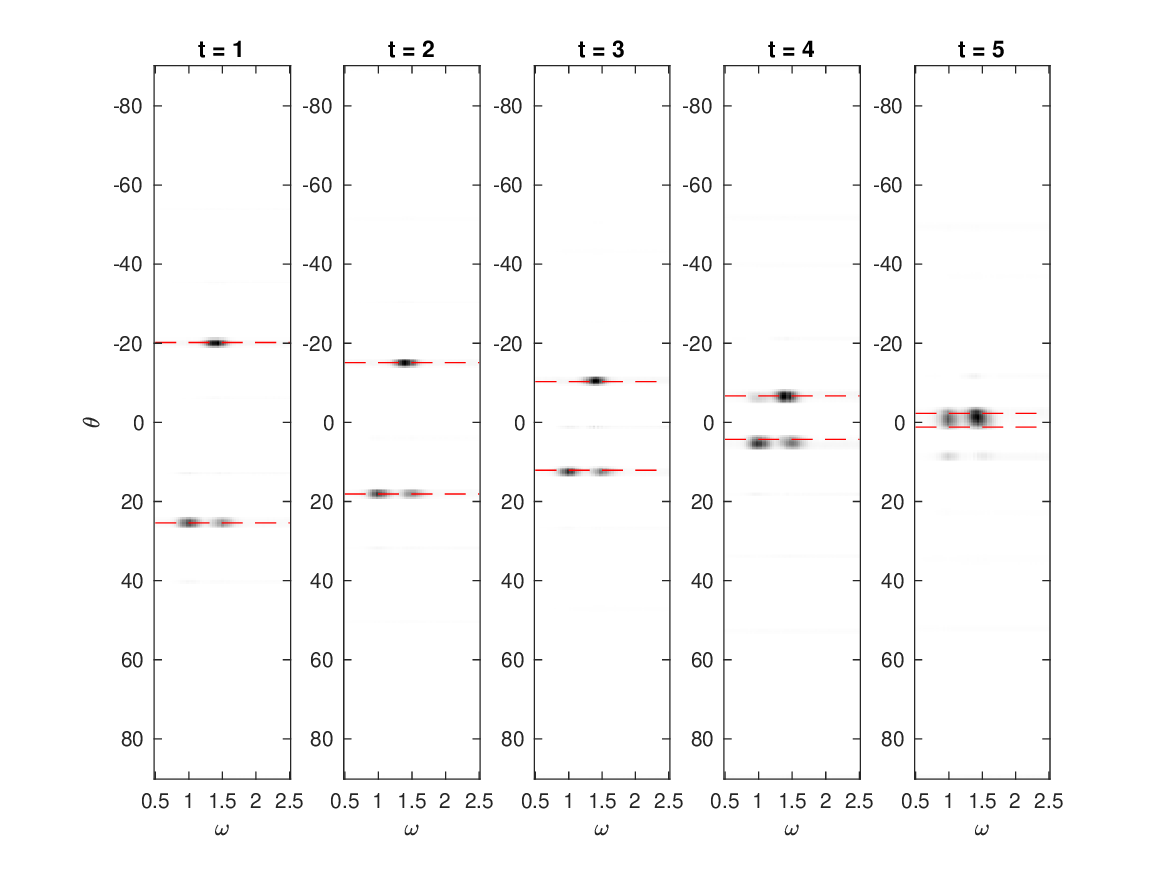}
           \caption{Group lasso estimate.}
            \label{fig:group_lasso_spectrum}
\end{subfigure}
%
\begin{subfigure}[t]{0.23\textwidth}
        \centering
            \includegraphics[trim={40pt 20pt 40pt 15pt}, clip, width=\textwidth]{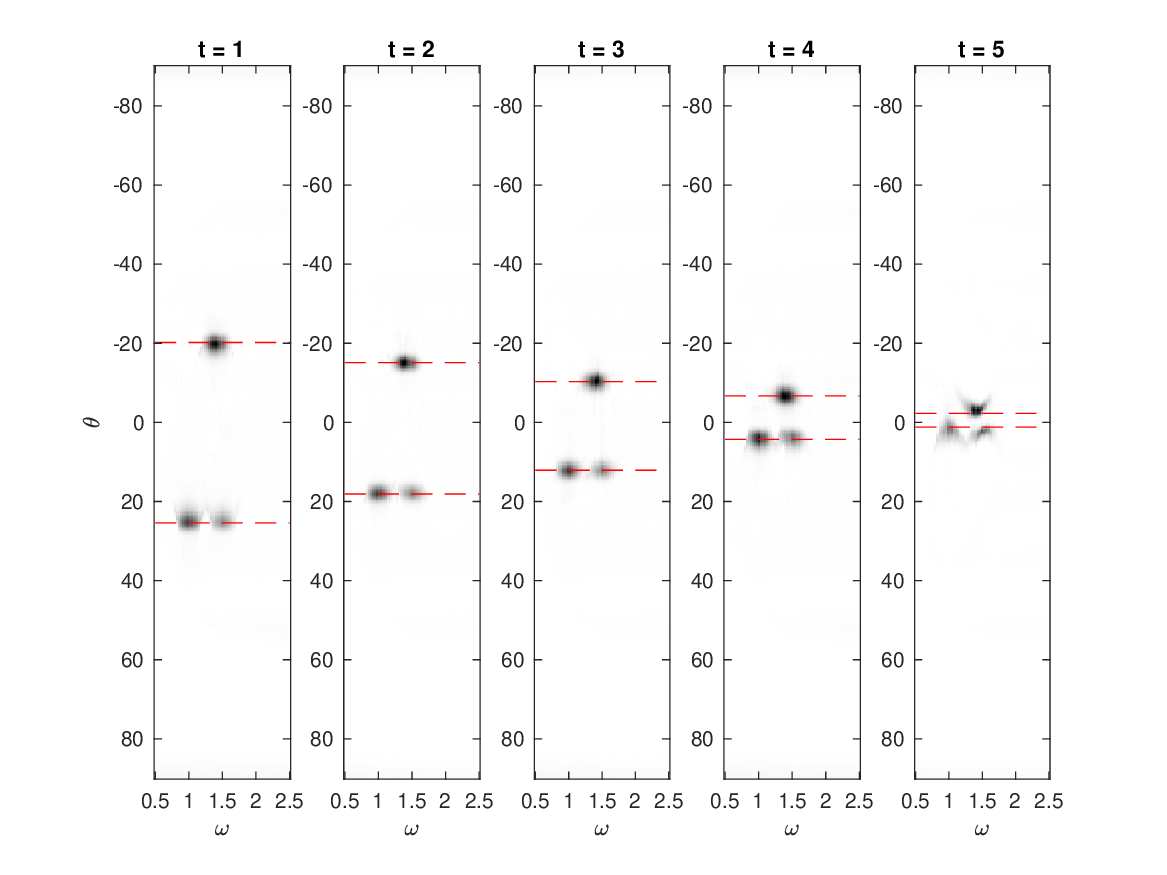}
           \caption{Proposed estimate.}
            \label{fig:ot_spectrum}
\end{subfigure}
\caption{Estimated spatio-temporal spectra in tracking scenario with two targets. }
\end{figure}

\section{Numerical experiments}
In this section, we illustrate the proposed method, and in particular the value of information sharing between frequencies and successive time-points as promoted by the group-sparsity promoting penalty of \eqref{eq:ot_discreg} and the OT distance, respectively. We do this in a simulated scenario as well as for real data measured on a hydrophone array. Throughout, we use a cost function according to \eqref{eq:cost_seq} with $c(\theta_{i_t},\theta_{i_t+1}) = (\theta_{i_t} - \theta_{i_t+1})^2$.
\subsection{Simulated two-target scenario}
Consider two broad-band point sources moving in angle space.
The (constant) ground truth temporal spectra are shown in Figure~\ref{fig:temporal_spectrum}. Using a uniform linear array consisting of $Q = 11$ sensors, we for $F = 63$ frequencies uniform on [0.5, 2.5] (in angular frequency) estimate the array covariance matrix at $\cT = 5$ time instances by the sample covariance matrix from 200 array snapshots. The array signals are contaminated by spatially and temporally white Gaussian noise as to yield a signal-to-noise ratio of 10 dB.
The estimated sequence of spatio-temporal spectra are shown in Figure~\ref{fig:ot_spectrum}, where ground-truth source locations are indicated by dashed lines. As can be seen, the proposed method is able to produce estimates indicating localized and well-separated sources. 
As reference, Figures~\ref{fig:capon_spectrum} and \ref{fig:group_lasso_spectrum} show estimates produced by the standard minimum-variance distortionsless response (MVDR/Capon) spatial spectral estimator \cite{Capon69}, and non-negative group lasso (with each group being the frequencies corresponding to a spatial angle), respectively. Note here that the MVDR estimate treats each time instance and frequency independently, whereas the group lasso fuses information across frequency by means of its sparsity penalty. It may be noted that neither of the reference methods manages to accurately separate the targets in angle and frequency at $t = 5$ due to the limited array aperture and the frequency overlap of the sources.
%
%
\begin{figure}[t]
\centering
\centering
            \includegraphics[width=.45\textwidth]{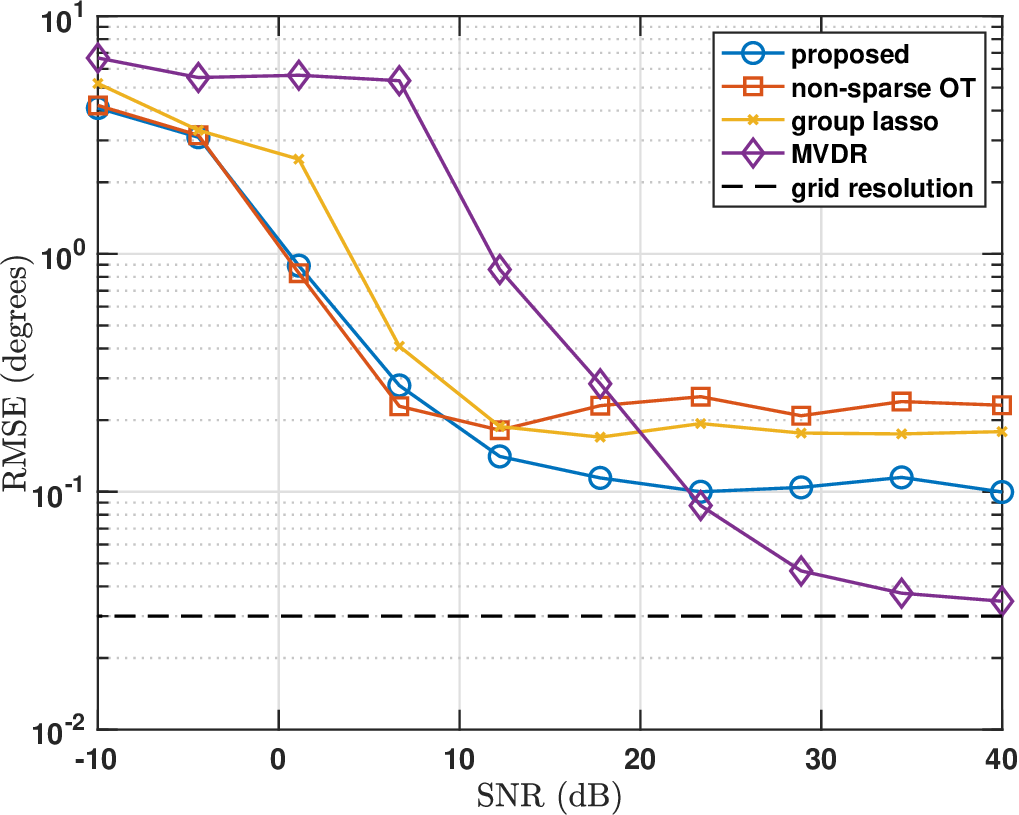}
           \caption{Target angle estimation error as a function of SNR.}
            \label{fig:simulation_study}
\end{figure}
%
\subsection{Localization accuracy}
For the same scenario, we study the accuracy of the target angle estimates as a function of the sensor noise. Figure~\ref{fig:simulation_study} shows the (root) means squared error (RMSE) for the angle estimates at time point $t = 4$, averaged over the two targets, for varying SNR. For each SNR, 50 Monte Carlo simulations are performed, where the target angles are perturbed randomly as to avoid biasing effects caused by the discrete grid. The RMSE is displayed for the proposed method, as well as for the MVDR estimator, and the method from \cite{elvander2020multi} that does not include any sparsity-promoting penalty. For all methods, the angle estimates are determined as the peaks of the spatial spectrum, i.e., the spatio-temporal spectrum averaged over frequency. As can be seen, the information sharing induced by the sparsity-promoting penalty leads to more accurate estimates as compared to only using the OT dynamics. It may also be noted that the OT-based methods incur a small bias, visible for the lowest noise levels, due to the tying together of consecutive time-points. For higher levels of noise, this is outweighed by the increased robustness.
\subsection{Hydrophone array measurements}
We here consider a real-world example with monitoring of a scene using an $Q=8$ element non-uniform linear hydrophone array with a total aperture of 2.08 meters. The data consists of a 35 seconds long recording, with a sampling frequency of 32 kHz. The signal source is a surface vessel moving in shallow water. We construct the band-pass signals by means of the STFT using a Hann window of length 0.2 seconds with 50\% overlap. The array covariance matrix is estimated in each frame using exponential averaging, resulting in a sequence of $\mathcal{T} = 73$ observation time points. We apply the proposed method using $F = 61$ frequencies in the interval 2 kHz -- 8 kHz. The array response vectors $\ba_f$ are constructed under the assumption of free-field propagation and targets in the far-field, with an assumed speed of sound in water of 1480 m/s. The resulting estimates are shown in Figure~\ref{fig:hydrophone}. Here, the estimated spatial spectrum over time (averaged over the frequencies) is shown in Figure~\ref{fig:hydrophone_ot_spatial}, whereas the estimated (assumed stationary) temporal spectrum is shown in Figure~\ref{fig:hydrophone_ot_temporal}. As can be seen, the spatial spectrum is well resolved, showing a single target trajectory. Overlayed in Figure~\ref{fig:hydrophone_ot_temporal} is the Burg temporal spectral estimate, computed from one of the hydrophone channels under assumption of stationarity. As can be seen, the proposed estimate coincides well with the Burg estimate. It may here be noted that no annotated ground truth for the data is available. The corresponding MVDR estimates are shown in Figures~\ref{fig:hydrophone_mvdr_temporal} and \ref{fig:hydrophone_mvdr_spatial}. As can be seen, the spatial spectrum is less well-resolved and contains more noise, and the corresponding temporal spectrum does not give an accurate representation of the signal's frequency content.
\begin{figure}[t]
\centering
%
\begin{subfigure}[t]{0.23\textwidth}
        \centering
            \includegraphics[width=\textwidth]{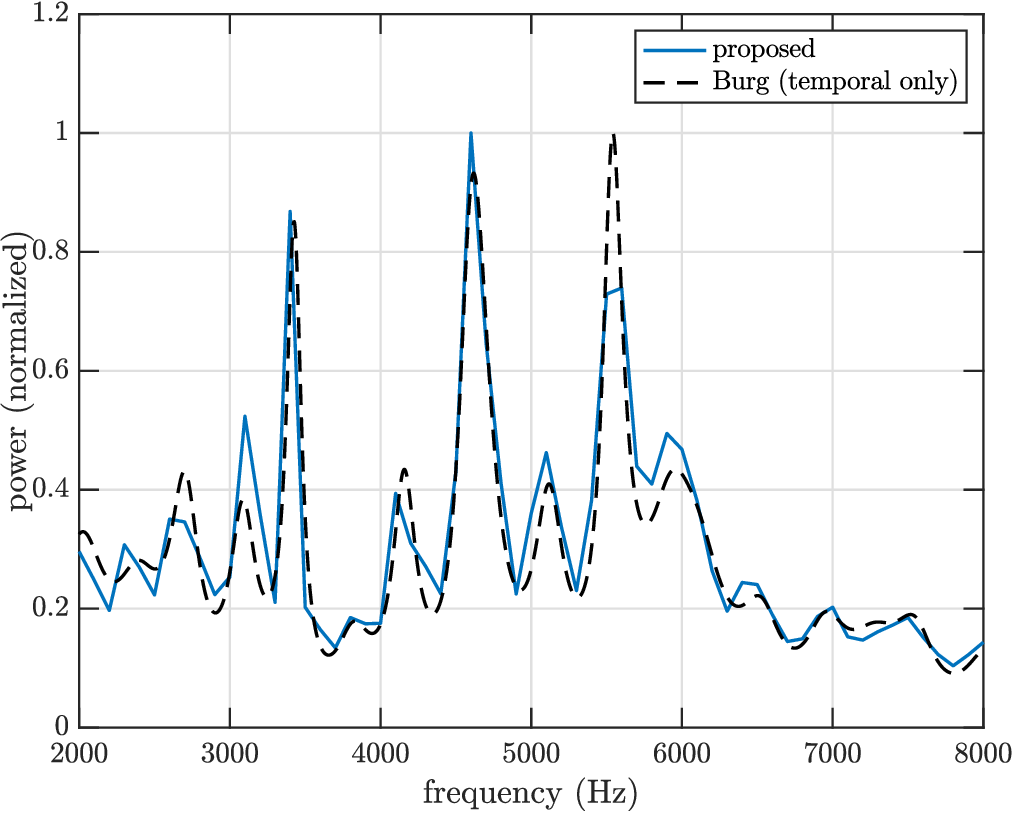}
            \caption{Proposed: temporal.}
            \label{fig:hydrophone_ot_temporal}
\end{subfigure}
\begin{subfigure}[t]{0.23\textwidth}
        \centering
            \includegraphics[width=\textwidth]{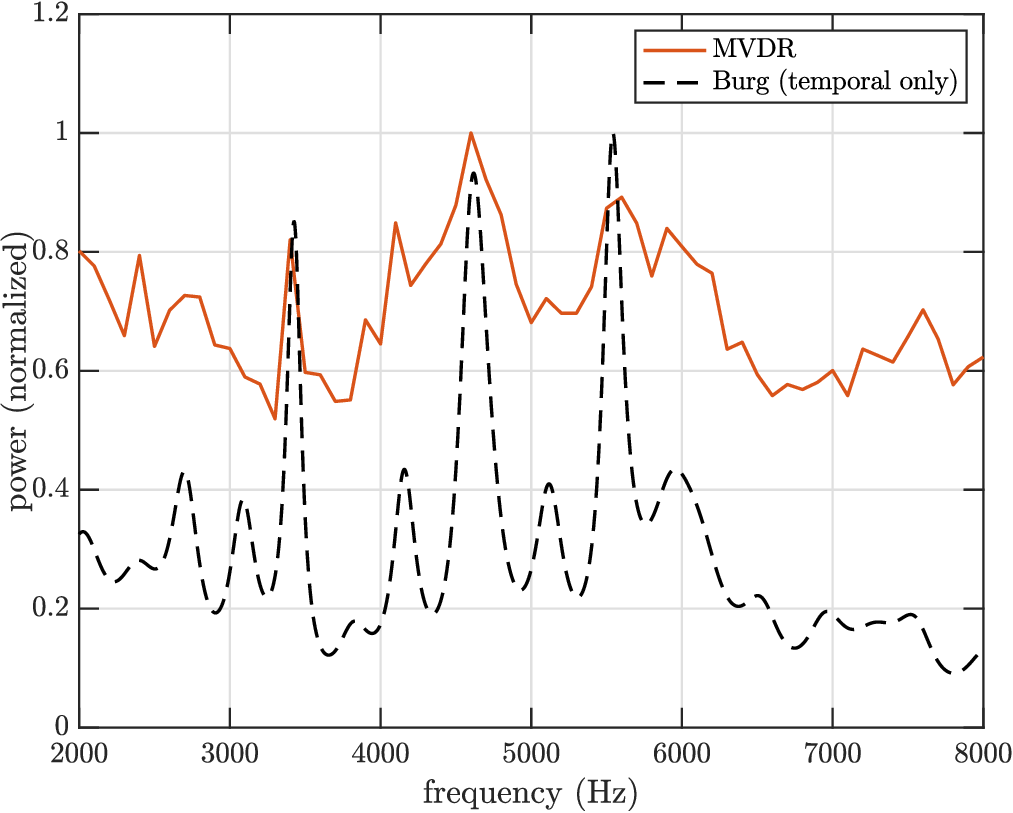}
           \caption{MVDR: temporal.}
            \label{fig:hydrophone_mvdr_temporal}
\end{subfigure}
%
\begin{subfigure}[t]{0.23\textwidth}
        \centering
            \includegraphics[width=\textwidth]{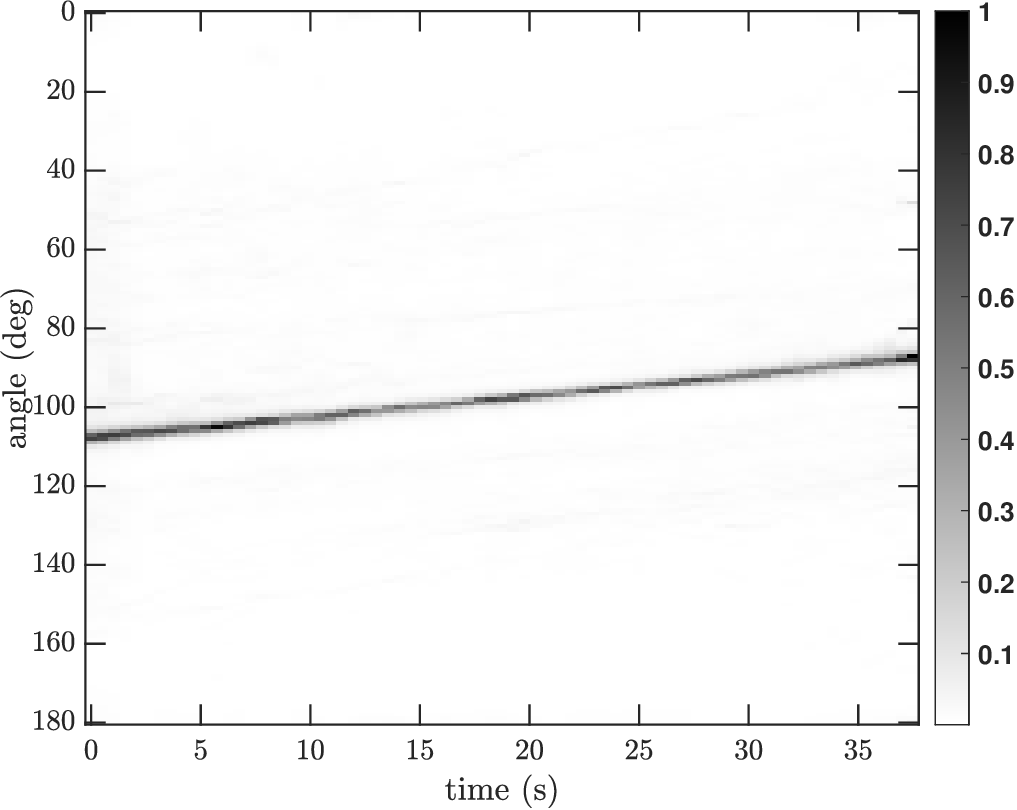}
           \caption{Proposed: spatial.}
            \label{fig:hydrophone_ot_spatial}
\end{subfigure}
%
\begin{subfigure}[t]{0.23\textwidth}
        \centering
            \includegraphics[width=\textwidth]{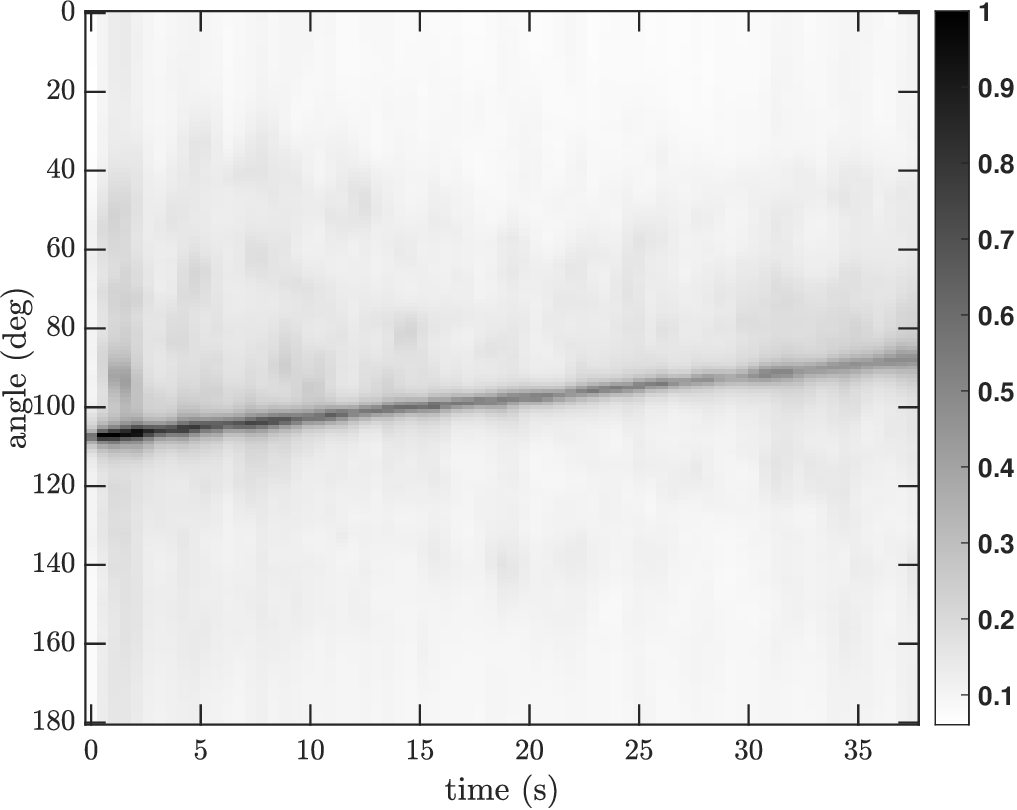}
           \caption{MVDR: spatial.}
            \label{fig:hydrophone_mvdr_spatial}
\end{subfigure}
\caption{Spectral estimates for hydrophone array data. }
\label{fig:hydrophone}
\end{figure}

\section{Conclusion}

In this work, we modeled multi-sensor broad-band signals by means of the concept of a spatio-temporal spectrum, and proposed to track the evolution of time-vaying spatio-temporal spectra by a group-sparse optimal transport formulation.
This allows us to fuse information across both separate time-instances and across frequency, and numerical experiments show that we achieve accurate estimates of the frequency content and location of broad-band signal sources.

\appendix

\subsection{Proof of Theorem~\ref{thm:dual}}

We introduce the auxiliary variables $\Delta_f^{(t)} =  \bG_f P_t(\bM_f) - \br_f^{(t)}$ and $\bPhi_f^{(t)} = P_t(\bM_f)$ for $f \in \indset{F}$, $t \in \indset{\cT}$ to define the Lagrangian 
\compress
\begin{equation*} 
\begin{aligned}
& \sum_{f=1}^F \langle \bC, \bM_f\rangle + \epsilon D(\bM_f) + \eta \sum_{t=1}^{T} \norm{\begin{bmatrix}\bPhi_1^{t},\ldots,\bPhi_f^{t} \end{bmatrix}}_{\infty,1} \\
& + \sum_{t=1}^T \sum_{f= 1}^F \bigg( \frac{\gamma}{2} \norm{\Delta_f^{(t)} }_2^2 +  (\lambda_f^{t})^T \left(\Delta_f^{(t)} -  \bG_f P_t(\bM_f) + \br_f^{(t)}\right) \\
& \hspace{50pt} +  (\bpsi_f^{t})^T \left(\bPhi_f^{(t)} - P_t(\bM_f) \right) \bigg).
\end{aligned} \compressmini
\end{equation*}
Minimizing this Lagrangian with respect to $\bM_f$ and $\Delta_f^{(t)}$ yields \eqref{eq:M} and $\Delta_f^{(t)} = -\frac{1}{\gamma} \lambda_f^{(t)}$ \cite[Proof of proposition~1]{elvander2020multi}.
For a fixed $t$, minimizing the Lagrangian with respect to $\bPhi_f^{(t)}$ for $f \in \indset{F}$ requires solving
\compress
\begin{equation*}
\minwrt[\bPhi_f^{(t)}, f \in \indset{F}] \ \ \eta \norm{\begin{bmatrix}\bPhi_1^{t},\ldots,\bPhi_f^{t} \end{bmatrix}}_{1,\infty} + \sum_{f=1}^F (\bpsi_f^{t})^T \bPhi_f^{(t)}. \compressmini
\end{equation*}
Note that this term can be treated separately for each element of the involved vectors. For each $i=1,\dots,N$, this becomes
\compress
\begin{equation*}
\begin{aligned}
& \infwrt[ {[\bPhi_f^{(t)}]_i, f \in \indset{F}} ]\ \  \eta\sup_{f \in \indset{F}} \{ [\bPhi_f^{(t)}]_i \}  + \sum_{f=1}^F ([\bpsi_f^{t}]_i)^T [\bPhi_f^{(t)}]_i\\
& = \begin{cases} 0, \text{ if } \sum_{f=1}^F (\bpsi_f^{(t)})_i \le \eta\\
-\infty, \text{ otherwise.}
\end{cases}
\end{aligned} \compressmini
\end{equation*}
Thus, in order for the dual to be bounded, the dual variables $\bpsi_f^{(t)}$ must satisfy $\norm{\begin{bmatrix}\bpsi_1^{(t)},\ldots,\bpsi_F^{(t)}\end{bmatrix}}_{1,\infty} \leq \eta $.
Plugging the optimal $\bM_f$ and $\Delta_f^{(t)}$ in the Lagrangian and using the constraints on $\bpsi_f^{(t)}$ results in the dual problem.

\subsection{Proof of Theorem~\ref{thm:psi}}

Note that problem \eqref{eq:update_psi} decouples in the vector indices. For each $i=1,\dots,N$, we solve a problem of the form
\begin{equation} \label{eq:psi_proof_opt} \compressmini
\begin{aligned}
 \minwrt[ \bx\in \RR^F ] \  \exp(- \bx )^T \bz , \; \text{subject to }  \| \bx \|_1   \leq p,
  \end{aligned} \compressmini
\end{equation}
where $\bx = - 1/\epsilon [ (\bpsi_1^{(t)})_i,\dots, (\bpsi_F^{(t)})_i ]^T$, $p=\eta / \epsilon$, and $\bz$ is defined as in step 1) of the theorem. As $\bz$ constructed in this way is non-negative, it is clear that the minimizer of \eqref{eq:psi_proof_opt} is non-negative and satisfies $  \| \bx \|_1  = p$.
Consider the Lagrangian of \eqref{eq:psi_proof_opt} with Lagrangian multiplier $\nu>0$, given by $\exp(- \bx )^T \bz + \nu(\| \bx \|_1 - p)$.
Let $\bs$ be a subgradient of $\|\bx\|_1$, then the Lagrangian is minimized with respect to $\bx$ if $- \exp(- \bx) \odot \bz + \nu \bs =0$.
From this we can conclude that 
the optimal $\bx$ is elementwise of the form $ x_f = \max(\log(z_f) - \log(\nu), 0)$.
We now find $\nu$ such that the optimal $\bx$ satisfies $ \| \bx \|_1  = p$. Therefore, note that the function
\compress
\begin{equation*}
g(\nu)  = \|\bx(\nu) \|_1 - p = -p + \sum_{f=1}^F \max\left(\log(z_f) - \log(\nu), 0\right) \compressmini
\end{equation*}
is continuous piecewise differentiable with non-differentiable points in $z_1,\dots,z_F$. Moreover, $g$ is strictly decreasing and thus has a root in the interval $(-\infty, \max \{  z_f , f \in \indset{F}\})$.
Letting the elements in $\bz$ be sorted in ascending order, the root lies in an interval $[z_f, z_{f+1}]$ for which $g(z_f)>0$ and $g(z_{f+1})\le0$.
As $g$ in this interval is given by \eqref{eq:g_func}, we get the root as 
\compress
\begin{equation*}
\nu = \exp \bigg(  \frac{1}{F-f} \Big( -c + \sum_{\ell = f+1}^F \log(z_\ell) \Big) \bigg). \compressmini
\end{equation*}
Step 4) in the theorem follows from plugging the root into $ x_f = \max(\log(z_f) - \log(\nu), 0)$.
\balance
\bibliographystyle{IEEEtran}

\bibliography{references}

\end{document}